\documentclass[a4paper]{eccomas_2022}
\usepackage{amsmath,amsthm,amsfonts,amssymb}
\usepackage{mathtools}
\usepackage{enumitem}\setitemize{leftmargin=5mm}
\usepackage{accents}
\usepackage{color}
\usepackage[usenames,dvipsnames,svgnames,table]{xcolor}
\usepackage{cite}
\usepackage{soul}
\usepackage{hyperref}
\usepackage[capitalize]{cleveref}
\usepackage{fancyhdr}
\usepackage{xfrac}

\title{Robust finite element discretizations for a simplified Galbrun's equation}

\author{Tilman Alemán$^1$, Martin Halla$^2$, Christoph Lehrenfeld$^2$ and Paul Stocker$^2$}

\heading{Tilman Alemán, Martin Halla, Christoph Lehrenfeld and Paul Stocker}

\address{
$^{1}$ Institut für Geometrie und Praktische Mathematik (IGPM), RWTH Aachen University\\
Templergraben 55, 52062 Aachen, Germany\\
e-mail: aleman@igpm.rwth-aachen.de, www.igpm.rwth-aachen.de
\and
$^{2}$ Institute for Numerical and Applied Mathematics, 
University of Göttingen\\
Lotzestr. 16-18, 37083 Göttingen, Germany\\
e-mail: \{m.halla,lehrenfeld,p.stocker\}@math.uni-goettingen.de, https://cpde.math.uni-goettingen.de
}

\keywords{Galbrun's equation, FEM, discontinuous Galerkin, Helmholtz decomposition}

\theoremstyle{plain}
\newtheorem{theorem}{Theorem}
\newtheorem{lemma}[theorem]{Lemma}
\theoremstyle{definition}
\newtheorem{definition}[theorem]{Definition}
\newtheorem{remark}[theorem]{Remark}

\newcommand{\ol}[1]{\overline{#1}}
\newcommand{\ull}[1]{\underline{#1}}

\newcommand{\bpm}{\begin{pmatrix}}
\newcommand{\epm}{\end{pmatrix}}
\newenvironment{equations}{\equation\aligned}{\endaligned\endequation}

\newcommand{\bflow}{\mathbf{b}}

\renewcommand{\u}{\mathbf{u}}
\renewcommand{\v}{\mathbf{v}}
\newcommand{\w}{\mathbf{w}}

\newcommand{\Hdiv}{H(\div)}

\newcommand{\Th}{{\mathcal{T}_h}} 
\newcommand{\Fh}{{\mathcal{F}_h}} 
\newcommand{\Xh}{\mathbb{X}_h}
\newcommand{\dom}{\Omega} 
\newcommand{\nx}{{\mathbf{n}}} 

\newcommand{\Tnorm}[1]{|\!|\!|#1|\!|\!|}

\newcommand{\sjump}[1]{[\![ #1 ]\!]}
\newcommand{\jump}[1]{[\![ #1 ]\!]_\nx}
\newcommand{\bjump}[1]{[\![ #1 ]\!]_{\bflow}}

\newcommand{\avg}[1]{\{\!\!\{#1\}\!\!\}}
\newcommand{\cbdm}[1]{c_{#1}^{\resizebox{0.45cm}{!}{\text{BDM}}}}
\newcommand{\Pibdm}{\Pi_{\resizebox{0.45cm}{!}{\text{BDM}}}}
\newcommand{\Pibdmdg}{\Pi_{\resizebox{0.45cm}{!}{\text{BDM}}}^{\resizebox{0.3cm}{!}{\text{DG}}}}
\newcommand{\cbdmdg}{c_{\resizebox{0.45cm}{!}{\text{BDM}}}^{\resizebox{0.3cm}{!}{\text{DG}}}}

\newcommand{\conv}{\partial_{\bflow}} 
\renewcommand{\div}{\operatorname{div}}

\DeclareMathOperator{\curl}{curl}

\newcommand{\inner}[1]{( #1 )}

\newcommand*{\norm}[1]{\|#1\|}

\newcommand\restr[2]{{ \left.\kern-\nulldelimiterspace #1 \vphantom{\big|} \right|_{#2} }}

\newcommand{\IF}{\mathbb{F}}

\newcommand{\IP}{\mathbb{P}}
\newcommand{\IQ}{\mathbb{Q}}
\newcommand{\IQh}{\mathbb{Q}_h}
\newcommand{\IR}{\mathbb{R}}

\newcommand{\IV}{\mathbb{V}}
\newcommand{\IVh}{\mathbb{V}_{\!h}}
\newcommand{\IW}{\mathbb{W}}
\newcommand{\IWh}{\mathbb{W}_{\!h}}
\newcommand{\IX}{\mathbb{X}}
\newcommand{\IXh}{\mathbb{X}_h}

\newcommand{\calF}{\mathcal{F}}

\newcommand{\calN}{\mathcal{N}}
\newcommand{\calO}{\mathcal{O}}

\newcommand{\calT}{\mathcal{T}}

\newcommand{\bb}{\mathbf{b}}

\newcommand{\bbf}{\mathbf{f}}

\newcommand{\bu}{\mathbf{u}}
\newcommand{\bv}{\mathbf{v}}
\newcommand{\bw}{\mathbf{w}}
\newcommand{\bx}{\mathbf{x}}

\definecolor{pscol}{rgb}{0.8,0,0}

\definecolor{clcol}{rgb}{0,0.5,0.6}

\newcommand{\old}[1]{}

\usepackage{stmaryrd} 
\usepackage{pgf,tikz}
\usetikzlibrary{spy,matrix, calc, arrows,shapes,decorations}
\usepackage{tkz-euclide}
\usetikzlibrary{positioning,arrows,calc,decorations.markings,math,arrows.meta}
\usepackage{pgfplots} 
\usepgfplotslibrary{groupplots}
\usepackage{csvsimple} 
\usepackage{siunitx} 
\usepackage{tabularx}
\usepackage{booktabs}
\usepackage{multirow}
\usepackage[normalem]{ulem}
\useunder{\uline}{\ul}{}

\pgfplotscreateplotcyclelist{paulcolors}{%
violet,every mark/.append style={solid,fill=violet},mark=square*,very thick,mark size=3pt\\%
teal,every mark/.append style={solid,fill=teal},mark=*,very thick,mark size=3pt\\%
orange,every mark/.append style={solid,fill=orange},mark=diamond*,very thick,mark size=3pt\\%
violet,densely dashed,every mark/.append style={solid,fill=violet},mark=square*,very thick,mark size=3pt\\%
teal,densely dashed,every mark/.append style={solid,fill=teal},mark=*,very thick,mark size=3pt\\%
orange,densely dashed,every mark/.append style={solid,fill=orange},mark=diamond*,very thick,mark size=3pt\\%
violet,dash dot,every mark/.append style={solid,fill=violet},mark=square*,very thick,mark size=3pt\\%
teal,dash dot,every mark/.append style={solid,fill=teal},mark=*,very thick,mark size=3pt\\%
orange,dash dot,every mark/.append style={solid,fill=orange},mark=diamond*,very thick,mark size=3pt\\%
}
\pgfplotscreateplotcyclelist{paulcolors2}{%
teal,every mark/.append style={solid,fill=teal},mark=*,very thick,mark size=3pt\\%
orange,every mark/.append style={solid,fill=orange},mark=diamond*,very thick,mark size=3pt\\%
teal,densely dashed,every mark/.append style={solid,fill=teal},mark=*,very thick,mark size=3pt\\%
orange,densely dashed,every mark/.append style={solid,fill=orange},mark=diamond*,very thick,mark size=3pt\\%
teal,dash dot,every mark/.append style={solid,fill=teal},mark=*,very thick,mark size=3pt\\%
orange,dash dot,every mark/.append style={solid,fill=orange},mark=diamond*,very thick,mark size=3pt\\%
}
\pgfplotscreateplotcyclelist{paulcolors1}{%
teal,every mark/.append style={solid,fill=teal},mark=*,very thick,mark size=3pt\\%
}
\pgfplotscreateplotcyclelist{paulcolorsfill}{%
teal,fill=teal,fill opacity=0.5,thick\\
orange,fill=orange,fill opacity=0.5,thick\\
violet,fill=violet,fill opacity=0.5,thick\\
}
\pgfplotscreateplotcyclelist{paulcolors4}{%
violet,every mark/.append style={solid,fill=violet},mark=square*,very thick,mark size=3pt\\%
teal,every mark/.append style={solid,fill=teal},mark=*,very thick,mark size=3pt\\%
orange,every mark/.append style={solid,fill=orange},mark=diamond*,very thick,mark size=3pt\\%
cyan,every mark/.append style={solid,fill=cyan},mark=star,very thick,mark size=3pt\\%
}

\pgfplotsset{
    discard if not/.style 2 args={
        x filter/.append code={
            \edef\tempa{\thisrow{#1}}
            \edef\tempb{#2}
            \ifx\tempa\tempb
            \else
                
            \fi
        }
    }
}
\pgfplotsset{compat=1.16}

\pgfplotsset{tick label style={font=\small},label style={font=\small},legend style={font=\small},}
\pgfplotsset{ width=.49\linewidth}

\newcommand{\logLogSlope}[5]
{

    \pgfplotsextra
    {
        \pgfkeysgetvalue{/pgfplots/xmin}{\xmin}
        \pgfkeysgetvalue{/pgfplots/xmax}{\xmax}
        \pgfkeysgetvalue{/pgfplots/ymin}{\ymin}
        \pgfkeysgetvalue{/pgfplots/ymax}{\ymax}

        \pgfmathsetmacro{\xArel}{#1}
        \pgfmathsetmacro{\yArel}{#3}
        \pgfmathsetmacro{\xBrel}{#1-#2}
        \pgfmathsetmacro{\yBrel}{\yArel}
        \pgfmathsetmacro{\xCrel}{\xArel}

        \pgfmathsetmacro{\lnxB}{\xmin*(1-(#1-#2))+\xmax*(#1-#2)} 
        \pgfmathsetmacro{\lnxA}{\xmin*(1-#1)+\xmax*#1} 
        \pgfmathsetmacro{\lnyA}{\ymin*(1-#3)+\ymax*#3} 
        \pgfmathsetmacro{\lnyC}{\lnyA-#4*(\lnxA-\lnxB)}
        \pgfmathsetmacro{\yCrel}{\lnyC-\ymin)/(\ymax-\ymin)} 

        \coordinate (A) at (rel axis cs:\xArel,\yArel);
        \coordinate (B) at (rel axis cs:\xBrel,\yBrel);
        \coordinate (C) at (rel axis cs:\xCrel,\yCrel);

        \draw[#5] (B) -- (C);
        \addlegendimage{#5};
    }
}

\newcommand{\vel}{\bflow}

\abstract{
Driven by the challenging task of finding robust discretization methods for Galbrun's equation, we investigate conditions for stability and different aspects of robustness for different finite element schemes on a simplified version of the equations. The considered PDE is a second order indefinite vector-PDE which remains if only the highest order terms of Galbrun's equation are taken into account. A key property for stability is a Helmholtz-type decomposition which results in a strong connection between stable discretizations for Galbrun's equation and Stokes and nearly incompressible linear elasticity problems.
}

\begin{document}


\section{Introduction}
In this work we consider the numerical discretization of a model problem motivated by time-harmonic Galbrun-type equations as arising in aeroacoustics \cite{galbrun1931propagation} and helioseismology \cite{hallahohage}. 
Discretization of Galbrun's equation poses significant challenges, especially in the presence of non-uniform flow standard finite element methods will not be robust, see \cite{chabassier}.

Galbrun's equation is a second order partial differential equation, that describes wave propagation in the presence of a steady background flow in a Newtonian fluid, see \cite{maeder202090} for an overview.
The wave propagation is described in terms of a Lagrangian perturbations $\u$, that is the displacement of individual fluid particles. 
With additional rotational and gravitational terms these equations are used to model stellar oscillations.

Instead of these much more complex problems we restrict to a model PDE problem that only keeps their highest order derivatives and introduces a simple zeroth order term:
\begin{equations} \label{eq:strongpde}
        -\nabla (\rho c_s^2~ \div \u) + \conv (\rho\conv \u) 
        -\norm{\vel}_\infty^2 \rho\u
        = \bbf &\quad \text{in} ~ \Omega, \quad         \nx \cdot \u=0 &\quad \text{on}~\partial\Omega.
\end{equations}
where $\conv\u=(\bflow\cdot\nabla)\u$ and $\rho, c_s, \vel$ are density, sound speed, and background flow with $\nx \cdot \vel = 0$ on $\partial\Omega$.
We assume that the conservation of mass $\div (\rho \bb ) = 0$ holds.
Here we treat (homogeneous) boundary conditions for the normal component.
The coupling with transparent boundary conditions, while physically relevant, is outside the scope of this article.

The challenging aspect that remains from the more complex Galbrun-type equations mentioned above is the interplay between two second order differential operators with different sign. Both have a large non-trivial kernel: The grad-div-type operator vanishes for all $\curl$ fields, the streamline diffusion operator for all functions that are constant along the trajectories of the background flow $\bb$. 

\paragraph*{Structure of the manuscript.}
In the remainder of this work we proceed in the following steps: First, to expose the structure of the PDE problem we show the well-posedness of a proper weak formulation in \cref{sec:weakform}. The crucial tool here is a Helmholtz-type decomposition. Second, the corresponding structure can be exploited to find suitable discretizations. 
To this end, in  \cref{sec:disc:cond} we formulate conditions on proper discretization schemes which are sufficient to obtain stability and quasi best approximation results in a suitable discrete energy norms. 
The derived framework allows for several compatible discretization schemes, we introduce four different schemes in \cref{sec:discussion}. 
In \Cref{sec:robustness} we examine major aspects of robustness for these methods in more detail. 
The analysis of the methods is continued in \Cref{sec:DG}, based on the sufficient conditions for stability, which we prepared in \Cref{sec:disc:cond} for a generic finite element method.
We conclude the manuscript with numerical examples 
in \cref{sec:numerics}.

\section{Weak formulation and well-posedness} \label{sec:weakform}
Our assumptions follow \cite{hallahohage}. 
Let $\dom\subset\IR^3$ be a bounded Lipschitz domain.
Let $c_s, \rho, \colon\dom\to\IR$ be continuous and such that
with constants $\ull{c_s}, \ol{c_s}, \ull{\rho}, \ol{\rho} > 0$ and $\bflow\in L^\infty(\dom,\IR^3)$ there holds
\begin{align}
\ull{c_s}\leq c_s \leq \ol{c_s}, \qquad
\ull{\rho}\leq \rho \leq \ol{\rho}, \qquad
\div(\rho\bflow)=0~\text{in }\Omega, \quad\text{and}\quad \bflow\cdot\nx=0~\text{on }\partial\Omega,
\tag{A1}
\end{align}
where $\nx$ denotes the outward pointing normal vector on $\partial\dom$.
This ensures that the distributional derivative operator $\conv\u=(\bflow\cdot\nabla)\u$ is well-defined for $\u\in L^2(\dom)$. 
Depending on the application $\rho$ and $c$ may vary significantly (helioseismology) or not (aeroacoustics).

In the following we will denote inner products (and corresponding norms) on a domain $S$ by $\inner{\cdot,\cdot}_S$ respectively $\norm{\cdot}_S$. For $S=\dom$ we skip the index in inner products, i.e.\ $\inner{\cdot,\cdot} = \inner{\cdot,\cdot}_\dom$. 

Next, we want to consider a proper weak formulation of the problem \eqref{eq:strongpde}. To this end we introduce a proper Hilbert space for the solution
\begin{equation*}
    \IX=\{\u\in L^2(\dom,\IR^3):\ \div\u\in L^2(\dom),\ \conv \u\in L^2(\dom,\IR^3),\ \u\cdot\nx=0 \text{ on } \partial \dom\}.
\end{equation*}
We note that $\IX$ this is the same space that is used in the well-posedness analysis in \cite{hallahohage} for the equations of stellar oscillations. 
A straight-forward weak formulation for $\bbf \in L^2(\dom,\IR^3)$ of \eqref{eq:strongpde} reads as: Find $\u\in \IX$ s.t. \\[-3.5ex] 
\begin{equation}\label{eq:weakpde}\vspace*{0.05cm}  
        -\overbrace{\left[\inner{\rho\conv \u,\conv \u'} 
            + \norm{\vel}_\infty^2 \inner{\rho\u,\u'}\right]
        }^{=:a(\u,\u')} + \overbrace{\inner{\rho c_s^2\div \u,\div \u'}}^{=:b(\u,\u')}
            =\inner{\bbf,\u'}\quad \forall \u'\in \IX.
\end{equation}
The bilinear forms $a(\cdot,\cdot)$ and $b(\cdot,\cdot)$ allow to define a suitable norm on $\IX$ with $\norm{\cdot}_\IX^2 := a(\cdot, \cdot)+b(\cdot, \cdot)$.
We note that $a(\cdot,\cdot)$ alone already defines an inner product on $\IX$ with associated norm $\norm{\cdot}_a$,
and we can therefore consider the following Helmholtz-type decomposition:
\begin{equations}\label{eq:hhdecomp}
    \IX = \IV \!\oplus_{a}\! \IW \text{ with }
    \IV := \ker b = \{ \v \in \IX \mid \div \v = 0\},  \IW := \{ \w \in \IX \mid a(\w,\v) = 0~\forall~\v\in \IV \}.
\end{equations}
It allows us to uniquely decompose every function $\u \in \IX$ as $\u = \v + \w$ with $\v \in \IV$, $\w \in \IW$ and 
\begin{equation}\label{eq:decomp}  
    \norm{\u}_\IX^2 = \norm{\v+\w}_\IX^2 = a(\v,\v) + a(\w,\w) + b(\w,\w).
\end{equation}
Plugging $\u = \v + \w$ and $\u' = \v' + \w'$ with $\v,\v' \in \IV$ and $\w,\w' \in \IW$ into \eqref{eq:weakpde} we obtain the two separated problems\vspace*{-0.2cm}   
\begin{subequations} \label{eq:hhdecomposed}
\begin{align}
     -a(\v,\v') &= \inner{\bbf,\v'} && \forall \v' \in \IV, \label{eq:hhdecomposed1}\\
     - a(\w,\w') + b(\w,\w') &= \inner{\bbf,\w'} && \forall \w' \in \IW. \label{eq:hhdecomposed2}
\end{align}
\end{subequations}
The problem on the divergence-free subspace $\IV$ is obviously well-posed as on $\IV$ we have $\norm{\v}_a = \norm{\v}_\IX$ for all $v \in \IV$ and all prerequisites of the Lax-Milgram theorem are fulfilled, so that $\Vert \v \Vert_a \leq \norm{\bbf}_{\IX'}$. It remains to investigate the problem on the subspace $\IW$. The next lemma treats this problem in a general setting that allows to generalize it to the discrete level below.

\begin{lemma} \label{lem:wellposed1}
Let $a(\cdot, \cdot)$ and $b(\cdot, \cdot)$ be continuous bilinear forms 
on the Hilbert space $(\IX, \norm{\cdot}_\IX)$. 
Assume $a(\cdot,\cdot)$ to be an inner product on $\IX$, such that we can decompose $\IX= \IV\oplus_a \IW$ as in \eqref{eq:hhdecomposed}. 
Let $a(\cdot,\cdot)$ be controlled by $b(\cdot,\cdot)$ on $\IW$, i.e.
\begin{align}\label{eq:control}
    a(\w, \w) \leq c_b^{-1} b(\w,\w) \quad \forall \w\in \IW.
\end{align}
for a constant $c_b>1$.
Then the problem 
\begin{align}\label{eq:SBWP}
        \text{Find } \u\in \IX \text{ s.t. }
        - a(\u,\u') + b(\u,\u') = \inner{\bbf,\u'}\quad \forall \u'\in \IX.
\end{align}
is inf-sup stable, with inf-sup constant $\hat{c}:=\frac{c_{b}-1}{c_{b}+1}$ there holds 
\begin{equation} \label{eq:infsup}
    \forall \u \in \IX, \exists \u' \in \IX, s.t. -a(\u,\u')+b(\u,\u') \geq \hat{c} \norm{\u}_{\IX} \norm{\u'}_{\IX}.
\end{equation}
\end{lemma}
\begin{proof}
    By definition of the norm, we have $\norm{\v}_\IX^2=a(\v,\v)$ for all $\v$ in $\IV=\ker b$, and thus problem \eqref{eq:hhdecomposed1} is well-posed. We turn to \eqref{eq:hhdecomposed2}. For $\gamma\in(0,1)$ we have that for all $\w \in \IW$
\begin{align*}
    - a(\w,\w) + b(\w,\w) &= -a(\w,\w)+ \gamma b(\w,\w) + (1-\gamma)b(\w,\w) \\
                          &\overset{\eqref{eq:control}}{\geq} (\gamma c_b-1)a(\w,\w) + (1-\gamma)b(\w,\w) 
                          \geq \min\{(\gamma c_b-1),(1-\gamma)\}\norm{\w}^2_\IX.
\end{align*}
We can choose $\gamma=\frac{2}{c_b+1}$ such that $\hat{c} = \min\{(\gamma c_b-1),(1-\gamma)\}=\frac{c_{b}-1}{c_{b}+1}$. 
Thus problem \eqref{eq:hhdecomposed2} is well-posed. Merging the well-posedness of both subproblems we obtain the inf-sup condition: For arbitrary $\u = \v + \w \in \IX$, $\v\in\IV$, $\w\in\IW$ with $\u' = - \v + \w \in \IX$ there holds
\begin{equation*}
    - a(\u,\u') + b(\u,\u') = a(\v,\v) - a(\w,\w) + b(\w,\w) \geq \hat{c} [ a(\v,\v) + a(\w,\w) + b(\w,\w) ] = \hat{c} \norm{\u}_{\IX}^2   
\end{equation*}
Finally, \eqref{eq:decomp} yields $\norm{\u}_{\IX} = \norm{\u'}_{\IX}$ so that \eqref{eq:infsup} follows.
\end{proof}
We now turn our attention to the crucial condition \eqref{eq:control}.
\begin{theorem}\label{thm:wellposed2}
    For $\dom$ a bounded Lipschitz domain the condition \eqref{eq:control} in \cref{lem:wellposed1} holds for $a(\cdot,\cdot)$ and $b(\cdot,\cdot)$  as in \eqref{eq:weakpde} under the assumption
\begin{align}
    \norm{c_s^{-1}\vel}_\infty^2\leq C^{-1} < 1,  \tag{A2}\label{A2}
\end{align}
for a constant $C>1$ that is only dependent on the ratios $R_\rho=\ull\rho^{-1}\ol\rho$,  $R_{c_s}=\ull{c_s}^{-1}\ol{c_s}$ and the domain $\Omega$.
\end{theorem}
\begin{proof}
First, we note that there holds $\norm{\v}_{a}\leq \sqrt{\ol{\rho}}    \norm{\vel}_\infty \norm{\v}_{H^1(\dom)}$ for $\v\in H^1_0(\dom,\mathbb{R}^d)\subset\IX$.
Now making use of classical inf-sup stability results for the divergence operator in $H^1_0$/$L^2_0$-settings such as they appear in Stokes-like problems, cf. e.g.\ \cite[Ch. III.6]{braess}
, we can deduce an inf-sup condition for $\tilde b(\v, q) =\inner{ \div \v,q}$ w.r.t.\ the $\norm{\cdot}_a$-norm :
 \begin{align}\label{eq:infsuparg}
    \sup_{\v \in \IX} \frac{\tilde b(\v,q)}{\norm{\v}_{a} }
     &  \geq \ol{\rho}^{-\frac12} \norm{\vel}^{-1}_\infty  \sup_{\v \in H^1_0(\dom,\mathbb{R}^d)} \frac{\inner{\div \v, q}}{ \norm{\v}_{H^1(\dom)} }
     \geq \ol{\rho}^{-\frac12} \norm{\vel}^{-1}_\infty  c_{\operatorname{div}} \norm{ q}_\Omega 
    \qquad \forall q\in L^2_0(\dom)
\end{align}
where $c_{\operatorname{div}}$ is the corresponding inf-sup constant of the divergence operator in the $H^1_0$/$L^2$-setting. 
From the inf-sup condition it directly follows, cf. e.g.\ \cite[Lemma III.4.2]{braess}, that $ \tilde{b}: \IW\to L^2_0, \w \mapsto  \div \w$ is an isomorphism and there holds $\norm{\div\w}_\Omega \geq    \ol{\rho}^{-\frac12} \norm{\vel}^{-1}_\infty   c_{\operatorname{div}}\norm{\w }_{a}~\forall \w \in \IW$  and hence
\begin{align*} 
    a(\w,\w) 
    \leq c_{\operatorname{div}}^{-2} R_{c_s}^2 R_\rho \norm{c_s^{-1}\vel}_\infty^{2} b(\w,\w) \quad\forall \w\in\IW.\\[-7.5ex] 
\end{align*}                   
\end{proof}
\begin{remark}
    We note that the smallness assumption on $c_s^{-1}\vel$ in the previous theorem depends also on the fluctuations of the parameter fields $\rho$ and $c_s$. This dependency could also be hidden in the inf-sup constant for a corresponding inf-sup condition with properly weighted norms. However, in view of the discrete setting, where we will also rely on an inf-sup constants of discretizations in the usual $H^1_0$/$L^2_0$-setting, we will stick with this more explicit though potentially less sharp estimate.  
\end{remark}

\section{Generic FEM discretizations and sufficient conditions for stability} \label{sec:disc:cond}

A generic finite element discretization of \eqref{eq:strongpde}, respectively \eqref{eq:weakpde}, replaces $\IX$ with a finite dimensional space $\IXh$ and replaces $a(\cdot,\cdot)$ and $b(\cdot,\cdot)$ by discrete counter parts:
\begin{equation}
    \framebox[0.9\textwidth][c]{
        $\begin{aligned}
        \text{Find } \u_h \in \IXh \text{ s.t. } - a_h(\u_h,\u_h') + b_h(\u_h,\u_h')
            =\inner{\bbf,\u_h'}\quad \forall \u_h'\in \IXh.
        \end{aligned}$
    }
\tag{M}
\label{eq:discreteform}
\end{equation}
We assume that $a_h(\cdot,\cdot)$ defines an inner product on $\IXh$ to that we can apply discrete Helmholtz-type decomposition similar to \eqref{eq:hhdecomp} which takes the form $\IXh = \IVh \oplus_{a_h} \IWh$ with 
\begin{equation}\label{eq:discretehhd}
     \IVh := \{ \v_h \in \IX_h \!\mid\! b_h(\v_h,\w_h) = 0~\forall \w_h \in \IX_h\}, \IWh := \{ w_h \in \IX_h \!\mid\! a_h(\w_h,\v_h) = 0~\forall\v_h\in \IVh \}.
\end{equation}
Every function $\u_h \in \IX_h$ can again be written as $\u_h = \v_h + \w_h$ with unique $\v_h \in \IVh$ and $\w_h \in \IWh$. Plugging $\u_h = \v_h + \w_h$ and $\u_h' = \v_h' + \w_h'$ with $\v_h,\v_h' \in \IVh$ and $\w_h,\w_h' \in \IWh$ into \eqref{eq:discreteform} we obtain the corresponding two separated discrete problems
\begin{subequations} \label{eq:discretehhdd}
\begin{align}
     - a_h(\v_h,\v_h') &= \inner{\bbf,\v_h'} && \forall \v_h' \in \IVh, \label{eq:discretehhdd1}\\
     - a_h(\w_h,\w_h') + b_h(\w_h,\w_h') &= \inner{\bbf,\w_h'} && \forall \w_h' \in \IWh. \label{eq:discretehhdd2}
\end{align}
\end{subequations} 
\old{
We assume $\IX_h$ to be equipped with a proper discrete norm $\norm{\cdot}_{\IX_h}$ w.r.t.\ which $a_h(\cdot,\cdot)$ and $b_h(\cdot,\cdot)$ are continuous. Then we are in a setting that allows to apply \cref{lem:wellposed1}. 
Hence, we need to check symmetry and coercivity of $a_h(\cdot,\cdot)$ w.r.t.\ a discrete norm that is potentially weaker than $\IX_h$, so that it defines an inner product on $\IX_h$.
Furthermore, we require the correspondence of \eqref{eq:control}, i.e.\ $a_h(\w_h, \w_h) \leq c_{b_h}^{-1} b_h(\w_h,\w_h) \quad \forall \w_h \in \IWh$ for a constant $c_{b_h} > 1$.
}
We are hence in a setting that allows to apply \cref{lem:wellposed1} to show discrete stability. 

\begin{lemma} \label{lem:wellposedh}
Let $a_h(\cdot, \cdot)$ and $b_h(\cdot, \cdot)$ be continuous bilinear forms w.r.t.\ the norm $\norm{\cdot}_{\IX_h}^2 := a_h(\cdot, \cdot)+b_h(\cdot, \cdot)$ 
on $\IX_h$.
Further, let $a_h(\cdot,\cdot)$ define an inner product on $\IXh$ and $a_h(\cdot,\cdot)$ be controlled by $b_h(\cdot,\cdot)$ on the subspace $\IWh$ as in \eqref{eq:control}, i.e.
\begin{equation}\label{eq:wellposedh}
a_h(\w_h, \w_h) \leq c_{b_h}^{-1} b_h(\w_h,\w_h) \quad \forall \w_h \in \IWh
\end{equation}
for a constant $c_{b_h} > 1$.
Then \eqref{eq:discreteform} is inf-sup stable, i.e.\ with ${\hat c}_h = (c_{b_h}\!\!-\! 1)/(c_{b_h} \!\!+\! 1)$ there is
\begin{equation} \label{eq:infsuph}
    \forall \u_h \in \IXh, \exists \u_h' \in \IXh, s.t. -a_h(\u_h,\u_h')+b_h(\u_h,\u_h') \geq \hat{c}_h \norm{\u_h}_{\IXh} \norm{\u_h'}_{\IXh}.
\end{equation}
\end{lemma}
\begin{proof}
Follows by application of \cref{lem:wellposed1}.
\end{proof}
Additional assumptions on consistency and continuity in stronger norms yields quasi best approximation results. 
\begin{lemma} \label{lem:cea}  Let the assumptions of \cref{lem:wellposedh} hold. Furthermore, let $\u \in \IX_{\text{reg}} \subset \IX$ be the unique solution to \eqref{eq:strongpde} where $\IX_{\text{reg}}$ is a subspace of $\IX$ with additional regularity, especially a bounded norm $\norm{\cdot}_{\IX_{\text{reg}}}$ stronger than $\norm{\cdot}_{\IX}$ and $\norm{\cdot}_{\IX_h}$. 
If $a_h(\cdot,\cdot)$ and $b_h(\cdot,\cdot)$ are also continuous in the first argument w.r.t.\ $\norm{\cdot}_{\IX_{\text{reg}}}$ for functions in $\IXh + \IX_{\text{reg}}$,
i.e.
$$
-a_h(\u+\u_h,\u_h')+b_h(\u+\u_h,\u_h') \leq c_{c} \Vert \u + \u_h \Vert_{\IX_{\text{reg}}} \norm{ \u_h'}_{\IXh} \qquad \forall~ \u_h, \u_h' \in \IXh,\ \forall \u\in\IX_{\text{reg}}
$$
and a Galerkin orthogonality holds, i.e.\ there is
$$
-a_h(\u,\v_h)+b_h(\u,\v_h) = \inner{\bbf, \v_h} \qquad \forall~\v_h \in \IX_h,
$$
then there additionally holds
\begin{align*}
	\norm{\u-\u_h}_{\IXh}\leq \inf_{\u_I \in \IX_h}\norm{\u-\u_I}_{\IXh} + \hat{c}_{h}c_c \norm{\u-\u_I}_{\IX_{\text{reg}}}.
\end{align*}
\end{lemma}
\begin{proof}
After decomposing the discretization error into $\norm{\u-\u_h}_{\Xh} \leq \norm{\u-\u_I}_{\Xh} + \norm{\u_I-\u_h}_{\Xh}$ for any $\u_I \in \IXh$, we derive a bound for $\norm{\u_I-\u_h}_{\Xh}$. Let $\u_h'$ be the supremizer to $\u_I-\u_h$ in \eqref{eq:infsuph}. Then, we apply inf-sup stability, consistency and boundedness yielding
\begin{align*}
    \hat{c}_h ~ \norm{\u_h - \u_I}_{\IXh} \norm{\u_h'}_{\IXh} & \leq -a_h(\u_h - \u_I,\u_h')+b_h(\u_h - \u_I,\u_h') \\
    & = -a_h(\u - \u_I,\u_h')+b_h(\u - \u_I,\u_h') \leq c_c \norm{\u - \u_I}_{\IX_{\text{reg}}}  \norm{\u_h'}_{\IXh}.\\[-7.5ex]  
\end{align*}
\end{proof}

\section{A set of finite element discretizations} \label{sec:discussion}
We introduce four different finite element based discretizations of problem \eqref{eq:weakpde}.
Starting from the generic finite element discretization stated in \eqref{eq:discreteform}, we choose $a_h(\cdot,\cdot),b_h(\cdot,\cdot), \text{ and }\IXh$ for each method.

The domain $\dom$ is subdivided in a non-overlapping shape regular simplicial mesh $\Th$.
We define the following dis- and continuous piecewise polynomial spaces with polynomials up to degree $p$
\begin{align}\label{eq:polspaces}
    \IQ^p(\Th)=\{u_h\in L^2(\Omega)\!: \restr{u_h}{T}\!\in\IP^p(T)\ \forall T\in\Th \},\quad  \IP^p(\Th) = \IQ^p(\Th)\cap C(\dom) 
\end{align}
We use the following notation for the mesh skeleton $ \Fh:=\cup_{T\in\calT_h}\partial T.$ 
On interior facets, i.e.\ $\Fh\setminus \partial\dom$, we define the following DG averages and jumps for a function $\u\in\IXh$
\begin{align*}
    \avg{\u}=\frac12(\u^++\u^-) && \bjump{\u}=\u^+(\vel\cdot\nx^+)+\u^-(\vel\cdot\nx^-) && \jump{\u}=\u^+\cdot \nx^+ + \u^-\cdot \nx^-  .
\end{align*}
Note that $\jump{\cdot}$ measures only the jump in normal direction while $\bjump{\cdot}$ measures the jump in all directions but is weighted with the normal velocity. Especially on facets with \(\vel \cdot \nx = 0\) (locally) the jump $\bjump{\cdot}$ vanishes (locally).   
On facets that coincide with the boundary, i.e. $\Fh\cap \partial\dom$, the averages and jumps are defined as
$\avg{\u}=\u$ and $\jump{\u}=\u\cdot \nx$. 
Finally, we introduce two discrete bilinear forms that contain standard symmetric interior penalty DG terms and Nitsche terms to enforce the boundary condition posed on the normal component of the discrete solution, while recalling that by assumptions $\nx\cdot\vel=0$ on the boundary,
\begin{align*}
    a^{\text{DG}}_h(\u_h, \u'_h) &:= \sum_{T\in\Th} \int_T \rho\left[\conv \u_h \cdot \conv \u'_h + \norm{\vel}_\infty^2 \u_h \cdot \u'_h \right] dx
            + \int_\Fh\rho \frac{\lambda_{\vel}}{h} \bjump{\u_h} \cdot  \bjump{\u'_h} ds\\
            &\qquad-\int_\Fh\rho \avg{\conv \u_h} \cdot \bjump{\u'_h} ds
            -\int_\Fh\rho \avg{\conv \u'_h} \cdot \bjump{\u_h} ds,\\
    b^{\text{DG}}_h(\u_h, \u'_h) &:= \sum_{T\in\Th} \int_T \rho c_s^2\div\u_h \div\u'_h \ dx 
    +\int_\Fh \rho c_s^2\frac{\lambda_{\nx}}{h} \jump{\u_h} \jump{\u'_h} \ ds \\
             &\qquad- \int_\Fh \rho c_s^2\avg{\div \u_h}\jump{\u'_h} \ ds 
            -\int_\Fh \rho c_s^2\avg{\div \u'_h} \jump{\u_h} \ ds.
\end{align*}
with $\lambda_\nx>0$ and $\lambda_\vel>0$.
To exclude errors stemming from boundary approximation, we use high order parametric mesh deformation techniques in the implementation.

Before we discuss four different discretizations, let us recall that the well-posedness of \eqref{eq:weakpde} is closely linked to the inf-sup stability for the divergence operator, as we have seen in the proof of \Cref{thm:wellposed2}.
This carries over to the discrete setting:
Crucial for a successful discretization will be to find a discrete counterpart to the inf-sup stability argument in \eqref{eq:infsuparg}. 
This is again closely linked to the difficulties appearing in the discretization of Stokes-like problems, so we can think of it this way:
For a chosen discrete space $\IX_h$, the problem is to find a stable Stokes pair $(\IX_h,\IQh)$ that fulfills a discrete version of \eqref{eq:infsuparg}, with a proper choice of a discrete divergence operator and pressure space $\IQh$. This motivates the following selection of methods.

\paragraph{M1: An $H^1$-conforming method}
is achieved by choosing finite elements that are continuous across element interfaces and the bilinear forms from problem \eqref{eq:weakpde} are left unaltered,
yielding a consistent method.
Furthermore, it is conforming so that $\IXh\subset \IX$. Summing up:
\vspace*{-0.1cm}
\begin{equation}
    \vspace*{-0.1cm}  
    \framebox[0.9\textwidth][c]{
        $\begin{aligned}
        &\IXh= [\IP^p(\Th)]^d,\quad a_h(\u_h,\u_h')=a(\u_h,\u_h'),\quad b_h(\u_h,\u_h')=b_h^{\text{DG}}(\u_h,\u_h')\\
        &\text{with $a(\cdot,\cdot)$ as in \eqref{eq:weakpde} of the continuous setting.}
        \end{aligned}$
    }
\tag{M1}\label{M1}
\end{equation}
Note that for $H^1$-conforming finite elements the interior penalty terms in the bilinear form $b_h^{\text{DG}}(\cdot,\cdot)$ vanish and only the Nitsche terms remain.
The method corresponds to the Stokes discretization  with  Scott-Vogelius elements given by the pressure space $Q_h=\IQ^{p-1}(\Th)$. This Stokes-pair is stable only for sufficiently large $p$.

\paragraph{M2: An $H^1$-conforming method in pseudo-pressure formulation.}
One way to obtain inf-stable Stokes elements for $H^1$-conforming velocities it to change the pressure space, e.g.\ as it is done in Taylor-Hood discretizations. The Taylor-Hood pair is given by choosing $Q_h=\IP^{p-1}(\Th)$ and is known to be stable for $p\geq 2$. This translates to the \emph{pseudo-pressure}  formulation:    
\vspace*{-0.1cm}
\begin{equation}
    \vspace*{-0.1cm}  
    \framebox[0.9\textwidth][c]{
        $\begin{aligned}
        &\IXh= [\IP^p(\Th)]^d,\quad a_h(\u_h,\u_h')=a(\u_h,\u_h'),\quad b_h(\u_h,\u_h')=b^{\text{pp}}_h(\u_h,\u_h')\\
        &\text{with }b^{\text{pp}}_h(\u_h, \u'_h) := \inner{\rho c_s^2\Pi\div\u_h, \Pi\div\u'_h}_\dom 
            +\inner{\rho c_s^2\frac{\lambda_{\nx}}{h} (\u_h\cdot\nx),(\u'_h\cdot\nx)}_{\partial\dom} \\
        &\hspace{3cm}- \inner{\rho c_s^2{\Pi\div \u_h}, (\u'_h\cdot\nx)}_{\partial\dom} 
             -\inner{\rho c_s^2{\Pi\div \u'_h}, (\u_h\cdot\nx)}_{\partial\dom}.
        \end{aligned}$
    }
\tag{M2}\label{M2}
\end{equation}
Here $\Pi$ is the weighted $L^2$-projection operator onto the space $\IP^{p-1}(\Th)$, with the weight $\inner{\rho c_s^{2}\cdot,\cdot}$. 
To implement the projection the problem can be reformulated using an auxiliary field, the \emph{pseudo-pressure} $p_h$, by computing
$(\u_h, p_h) \in \IXh\times \IP^{p-1}(\Th)$ such that
\begin{align*}
    -a_h(\u_h, \u_h')+\inner{\rho c_s^{2} \div\u_h',p_h}_\dom
    +\inner{\rho c_s^2 (\frac{\lambda_{\nx}}{h} (\u_h\cdot\nx)-p_h),(\u'_h\cdot\nx)}_{\partial\dom} 
    &= \inner{\bbf,\u_h'} &&\forall \u_h' \in \IXh, \\
    \inner{\rho c_s^{2} \div\u_h, q_h}_\dom - \inner{\rho c_s^{2} p_h,q_h}_\dom 
            - \inner{\rho c_s^2(\u_h\cdot\nx), q_h}_{\partial\dom} 
    &=0 &&\forall q_h \in \IP^{p-1}(\Th).
\end{align*}

\paragraph{M3: An $\Hdiv$-conforming discontinuous Galerkin method.}
For this method we choose $\Xh$ to be piecewise polynomials that are normal (but not tangential) continuous across inter-element boundaries. 
The approximation space has less regularity than $\IX$, making the method non-conforming.
Possible choices for the basis of $\Xh$ are the Raviart–Thomas or Brezzi-Douglas-Marini finite element, which are known to provide inf-sup stable Stokes discretizations if combined with proper pressure spaces and DG bilinear forms and norms. 
To account for the missing regularity we use a symmetric interior penalty formulation for the streamline diffusion. 
\vspace*{-0.1cm}  
\begin{equation}
    \vspace*{-0.1cm}  
    \framebox[0.9\textwidth][c]{
        $\begin{aligned}
        &\IXh=\{\u_h\in[\IQ^{p}(\Th)]^d,\ \jump{\u_h}=0 \text{ on all } F\in \Fh \} = [\IQ^{p}(\Th)]^d\cap\Hdiv \\
        &a_h(\u_h,\u_h')=a^{\text{DG}}_h(\u_h,\u_h'),\quad b_h(\u_h,\u_h')=b(\u_h,\u_h') \\
        \end{aligned}$
    }
\tag{M3}\label{M3}
\end{equation}
Note that the jumps $\bjump{\cdot}$ are effectively only jumps in the tangential direction as the discrete functions are all normal-continuous.  

\paragraph{M4: A discontinuous Galerkin method.}
We can further reduce regularity across elements considering a fully discontinuous scheme.
Compared to \eqref{M3}, this requires additional penalty terms for the divergence.
A symmetric interior penalty discontinuous Galerkin discretization is then given by
\vspace*{-0.1cm}  
\begin{equation}
    \vspace*{-0.1cm}  
    \framebox[0.9\textwidth][c]{
        $\begin{aligned}
        &\IXh=\{\u_h\in L^2(\Omega,\IR^d):\ \restr{\u_h}{T}\in\IP^p(T,\IR^d)\ \forall T\in\Th \}\\
        &a_h(\u_h,\u_h')=a^{\text{DG}}_h(\u_h,\u_h'),\quad b_h(\u_h,\u_h')=b^{\text{DG}}_h(\u_h,\u_h') \\
        \end{aligned}$
    }
\tag{M4}\label{M4}
\end{equation}
This method is again inf-sup stable, 
however, it requires more degrees of freedom then the $\Hdiv$-conforming scheme.

\section{Aspects of robustness} \label{sec:robustness}
In this section we want to focus on two additional aspects of robustness which are related to corresponding concepts in linear elasticity problems \cite{FLLS_ARXIV_2020}: volume-locking and gradient-robustness.
We will introduce these concepts and illustrate them on two comparably simple numerical examples for the method \eqref{M1}--\eqref{M4}
\footnote{The numerical examples can be reproduced with the source code and reproduction instructions published at \cite{data}.}.

Let us start with the definition of a volume-locking free discretization which addresses the case $c_s\to \infty$, i.e. the case where the solution ends up in the divergence-free subspace $\IV$.
\begin{definition}\label{def:lf}  
    We call a discretization \eqref{eq:discreteform} \emph{volume-locking free} if for a divergence-free solution $\v \in \IV$ the discretely divergence-free space $\IVh$ provides optimal approximation properties, i.e. 
   \begin{equation}
       \inf_{\v_h^* \in \IVh} \norm{\v - \v_h^*}_{\IXh} \leq C \inf_{\u_h^* \in \IXh} \norm{\v - \u_h^*}_{\IXh} 
   \end{equation}  
   for a constant $C$ that is independent of $h$ and $c_s$.
\end{definition} 
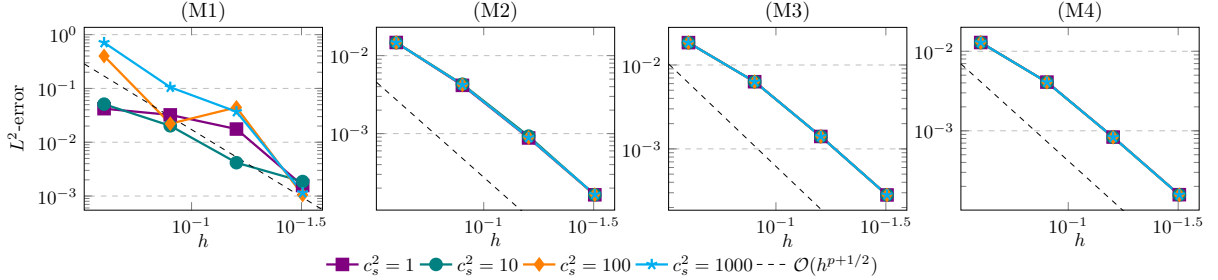
\begin{figure}[!ht]
\begin{center}
    \resizebox{\linewidth}{!}{
    \begin{tikzpicture} [spy using outlines={circle, magnification=4, size=2cm, connect spies}]
        \begin{groupplot}[%
          group style={%
            group name={my plots},
            group size=4 by 1,
            xlabels at=edge bottom,
            ylabels at=edge left,
          },
        legend style={
            legend columns=6,
            at={(-1.5,-0.2)},
            anchor=north,
            draw=none
        },
        xlabel={$h$},
        xlabel style={yshift=1.5ex},
        ylabel={$L^2$-error},
        ymajorgrids=true,
        grid style=dashed,
        cycle list name=paulcolors4,
        width=6cm,height=5cm,
        title style={yshift=-1.5ex},
        ]      
        \nextgroupplot[ymode=log,xmode=log,x dir=reverse, title={\eqref{M1}}]
        \foreach \cs in {1,10,100,1000}{
                \addplot+[discard if not={cs}{\cs}] table [x=h, y=errorH1, col sep=comma] {results/locking.csv}; 
            }
           \logLogSlope{1}{1}{0.8}{2.5}{black,dashed};
        \nextgroupplot[ymode=log,xmode=log,x dir=reverse, title={\eqref{M2}}]
        \foreach \cs in {1,10,100,1000}{
                \addplot+[discard if not={cs}{\cs}] table [x=h, y=errorH1pp, col sep=comma] {results/locking.csv}; 
            }
           \logLogSlope{1}{1}{0.7}{2.5}{black,dashed};
            \nextgroupplot[ymode=log,xmode=log,x dir=reverse, title={\eqref{M3}}]
        \foreach \cs in {1,10,100,1000}{
                \addplot+[discard if not={cs}{\cs}] table [x=h, y=errorHdiv, col sep=comma] {results/locking.csv}; 
            }
           \logLogSlope{1}{1}{0.8}{2.5}{black,dashed};
        \nextgroupplot[ymode=log,xmode=log,x dir=reverse, ,title={\eqref{M4}}]
        \foreach \cs in {1,10,100,1000}{
                \addplot+[discard if not={cs}{\cs}] table [x=h, y=errorDG, col sep=comma] {results/locking.csv}; 
            }
           \logLogSlope{1}{1}{0.8}{2.5}{black,dashed};
       \legend{$c_s^2=1$, $c_s^2=10$, $c_s^2=100$, $c_s^2=1000$, $\mathcal O(h^{p+1/2})$}
        \end{groupplot}
    \end{tikzpicture}}
\end{center}
    \vspace*{-0.5cm} 
    \caption{Study of volume-locking. We consider different values of $c_s$ under mesh refinement with fixed polynomial degree $p=2$.  }
    \label{fig:locking}
\end{figure}
To investigate volume-locking numerically for the discretizations \eqref{M1}--\eqref{M4} we consider the divergence-free solution 
$ \u=
\cos(\pi(x^2+y^2))(-y,x)^T 
$
on a disc with radius 1.
Fixing $\lambda_\vel=\lambda_\nx=10p^2, \vel=0.1(-y,x)^T$ and $\rho=1$ we vary $c_s^2=1,10,100,1000$.
We prescribe homogeneous boundary conditions
and choose $\bbf$ such that $\u$ is the solution. Note that neither the solution $\u$ nor $\bbf$ depend on $c_s$.
In \Cref{fig:locking} we compare convergence in the $L^2(\dom)$-norm in terms of the mesh size $h$ for different values of $c_s$.
Method \eqref{M1} shows signs of volume locking, the discretization is very sensitive to variations in the sound speed and convergence deteriorates for increasing $c_s$. 
Methods \eqref{M2}--\eqref{M4} are free of volume locking. 

Let us now turn our attention to gradient-robustness where the question of the response to gradient forces in the limit $c_s\to\infty$ is addressed. 
Gradient-robustness is linked to a compatibility of the decompositions \eqref{eq:hhdecomposed} and \eqref{eq:discretehhdd}.
If $\bbf = \nabla \phi$ we obtain from partial integration $\inner{\bbf, \bv'} = \inner{\nabla \phi, \bv'} = \inner{- \phi, \div \bv'} = 0$ for all $\bv' \in \IV$. 
Hence $\bu=\bw\in \IW $ and 
testing \eqref{eq:weakpde} with $\bw'=\bw$ and exploiting \eqref{eq:control} we get $\norm{c_s\sqrt{\rho} \div \bw}_\dom^2 \leq (1+c_b^{-1})^{-1} \norm{\phi}_\dom\norm{\div\bw}_\dom$ and therefore $\norm{\div \bw}_\dom=\calO(\ull{c_s}^{-2})$.
From \eqref{eq:infsuparg} we know that $\norm{\bu}_{\IX} \rightarrow 0$ as $\ull{c_s}\rightarrow\infty$. This motivates the following definition:
\begin{definition}\label{def:gr}  
If $\bbf = \nabla \phi$ for $\phi \in H^1(\dom)$, then for $ \ull{c_s}\to\infty$ and $\bu$ the solution to \eqref{eq:strongpde}
there holds $\Vert \bu\Vert_{\IX}\to 0$.
If for the discrete solution $\u_h\in\IXh$ to \eqref{eq:discreteform}   
there also holds $\Vert \bu_h \Vert_{\IXh}\to 0$ we call the discretization \eqref{eq:discreteform} \emph{gradient-robust}.     
\end{definition} 
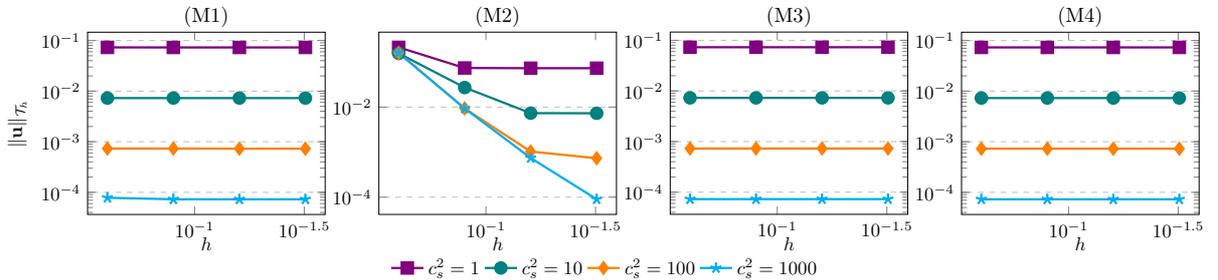
\begin{figure}[!ht]
\begin{center}
    \resizebox{\linewidth}{!}{
    \begin{tikzpicture} [spy using outlines={circle, magnification=4, size=2cm, connect spies}]
        \begin{groupplot}[%
          group style={%
            group name={my plots},
            group size=4 by 1,
            xlabels at=edge bottom,
            ylabels at=edge left,
          },
        legend style={
            legend columns=6,
            at={(-1.5,-0.2)},
            anchor=north,
            draw=none
        },
        xlabel={$h$},
        xlabel style={yshift=1.5ex},
        ylabel={$\norm{\u}_\Th$},
        ymajorgrids=true,
        grid style=dashed,
        cycle list name=paulcolors4,
        width=6cm,height=5cm,
        title style={yshift=-1.5ex},
        ]      
            \nextgroupplot[ymode=log,xmode=log,x dir=reverse, title={\eqref{M1}}]
            \foreach \cs in {1,10,100,1000}{
                \addplot+[discard if not={cs}{\cs}] table [x=h, y=normH1, col sep=comma] {results/rob.csv}; 
            }
            \nextgroupplot[ymode=log,xmode=log,x dir=reverse, title={\eqref{M2}}]
            \foreach \cs in {1,10,100,1000}{
                \addplot+[discard if not={cs}{\cs}] table [x=h, y=normH1pp, col sep=comma] {results/rob.csv}; 
            }
            \nextgroupplot[ymode=log,xmode=log,x dir=reverse, title={\eqref{M3}}]
            \foreach \cs in {1,10,100,1000}{
                \addplot+[discard if not={cs}{\cs}] table [x=h, y=normHdiv, col sep=comma] {results/rob.csv}; 
            }
        \nextgroupplot[ymode=log,xmode=log,x dir=reverse,title={\eqref{M4}}]
            \foreach \cs in {1,10,100,1000}{
                \addplot+[discard if not={cs}{\cs}] table [x=h, y=normDG, col sep=comma] {results/rob.csv}; 
            }
        \legend{$c_s^2=1$,$c_s^2=10$,$c_s^2=100$,$c_s^2=1000$,}
        \end{groupplot}
    \end{tikzpicture}}
\end{center}
    \vspace*{-0.5cm}  
    \caption{Study of gradient-robustness. We consider different values of $c_s$ under mesh refinement with $p=3$, for each method.}
    \vspace*{-0.6cm}  
    \label{fig:gradrob}
\end{figure}
To check numerically if the different methods exhibit gradient-robustness we take $\bbf=\nabla\phi$ with $\phi=x^6+y^6$ and homogeneous boundary conditions.
Thus for $c_s\rightarrow \infty$ we have $\u\rightarrow 0$.
We fix polynomial order $p=3$ and take the same parameters as previously, varying $c_s^2=1,10,100,1000$.
Gradient-robust methods will exhibit $\u_h\rightarrow 0$ as $c_s\rightarrow \infty$ independently of the mesh size.
\Cref{fig:gradrob} shows that only the $H^1$-conforming method with pseudo-pressure formulation is not gradient-robust.
This is no surprise -- the associated help problem is based on a Taylor-Hood velocity-pressure pair, which is known not to be gradient-robust, cf. \cite{10.1137}.

\section{Error analysis for a DG scheme} \label{sec:DG}
In this section we analyse the DG scheme \eqref{M4}. We will see that the analysis automatically covers the cases of the $H(\div,\Omega)$-conforming DG scheme \eqref{M3} as well as the $H^1$-conforming scheme \eqref{M1} in the stable Scott-Vogelius configuration, cf. \cref{rem:carryover} below. We do not treat the methods \eqref{M2} and \eqref{M1} in the general case due to their deficiencies observed before.

\begin{theorem}
    Let $\mathcal T_h$ be a shape regular mesh, i.e.\ $h_K/\rho_K\leq c_{\mathcal T}$. 
    Assume that 
    \begin{equation}
        \vel \text{ is Lipschitz with Lipschitz constant } L. \tag{A3}\label{A3}
    \end{equation}
    Then $\exists c_1,c_2>0$ such that with $\norm{\cdot}_{\frac12,h,\Fh} :=\norm{h_F^{-\frac12}\cdot }_{\Fh}$ (where $h_F$ is the diameter of a facet)
    \begin{equations}\label{eq:dgnorms}
        &a_h(\u_h,\u_h)\geq c_1  ( \norm{\sqrt{\rho}\conv \u_h}_{\Th}^2 + \norm{\sqrt{\rho}\u_h}_{\Th}^2 + \norm{\sqrt\rho  \bjump{\u_h}}_{\frac12,h,\Fh}^2 )\\
        &b_h(\u_h,\u_h)\geq c_2 ( \norm{c_s\sqrt\rho\div\u_h}_\Th^2 + \norm{c_s\sqrt\rho\jump{\u_h}}_{\frac12,h,\Fh}^2 )
    \end{equations}
    for $\lambda_\vel,\lambda_\nx$ large enough.
    Additionally
        $\Tnorm{\u_h}^2_{\IX_h} := a_h(\u_h,\u_h)+b_h(\u_h,\u_h)$
    defines a norm.
\end{theorem}
\begin{proof}
    Let $\bx_T$ denote the center of an element $T\in\Th$ and $\tilde\vel=\vel(\bx_T)$.
    Thanks to \eqref{A3}, we have that
    $ \norm{\vel-\tilde\vel}_T\leq Lh.$
    We will use this to reduce the convection operator to a discrete setting so we can apply the
    discrete trace inequality $\norm{\u_h}_{\partial K}\leq c_{\mathrm{tr}}ph^{-\sfrac12}\norm{\u_h}_K$ that holds for any discrete $\u_h\in\IP^p(K)$.
    Using additionally the inverse inequality $\norm{\nabla\u_h}_{K}\leq c_{\mathrm{inv}}h^{-1}\norm{\u_h}_K$ we get that 
    for $\u_h\in\IP^p(K)$ we have 
    \begin{align*}
        h^{\sfrac12}\norm{\sqrt\rho\conv\u_h}_{\partial T} 
        &\leq \sqrt{\ol\rho}\left(h^{\sfrac12}\norm{((\vel-\tilde\vel)\cdot \nabla)\u_h}_{\partial T} + h^{\sfrac12} \norm{(\tilde\vel\cdot\nabla)\u_h}_{\partial T}\right)\\
        &\leq c_{\mathrm{tr}}p\sqrt{\ol\rho}\left(Lh\norm{\nabla\u_h}_{T} +  \norm{(\tilde\vel\cdot\nabla)\u_h}_{T}\right)\\
        &\leq c_{\mathrm{tr}}p\sqrt{\ol\rho}\left(c_{\mathrm{inv}}L\norm{\u_h}_{T} +  \norm{((\tilde\vel-\vel)\cdot\nabla)\u_h}_{T} + \norm{(\vel\cdot\nabla)\u_h}_{T}\right)\\
        &\leq c_{\mathrm{tr}}p\sqrt{R_\rho}\left(2c_{\mathrm{inv}}L\norm{\sqrt\rho\u_h}_{T} + \norm{\sqrt\rho(\vel\cdot\nabla)\u_h}_{T}\right).
    \end{align*}
    And therefore
\begin{align*}
    a&_h^{\text{DG}}(\u_h,\u_h)
         \geq \norm{\sqrt{\rho}\conv \u_h}_{\Th}^2+ \norm{\vel}_\infty^2\norm{\sqrt{\rho}\u_h}_{\Th}^2 - \frac{h}{a\lambda_\vel}\norm{\sqrt\rho\avg{\conv \u_h}}_\Fh^2 +  (1-a)\frac{\lambda_\vel}{h}\norm{\sqrt\rho  \bjump{\u_h}}_\Fh^2 \\
          &\geq (1-\frac{c_{\mathrm{tr}}^2p^2R_\rho}{a\lambda_\vel}) \norm{\sqrt{\rho}\conv \u_h}_{\Th}^2
          + (\norm{\vel}_\infty^2-\frac{4c_{\mathrm{inv}}^2L^2c_{\mathrm{tr}}^2p^2R_\rho}{a\lambda_\vel})\norm{\sqrt{\rho}\u_h}_{\Th}^2+  (1-a)\frac{\lambda_\vel}{h}\norm{\sqrt\rho  \bjump{\u_h}}_\Fh^2
\end{align*}
for $0.5<a<1$.
With $h^{\sfrac12}\norm{c_s\sqrt\rho\div\u_h}_{\partial T} 
    \leq c_{\mathrm{tr}}p R_{c_s}\sqrt{R_\rho}\norm{c_s\sqrt\rho\div\u_h}_{T}$
    the result for $b_h^{DG}(\cdot,\cdot)$ follows analogously.
\end{proof}

\begin{theorem} 
    For $\dom$ being a polygonal domain or a domain with smooth boundary 
    the condition \eqref{eq:wellposedh} in \cref{lem:wellposedh} holds for \eqref{M4} under assumption \eqref{A2} (for a new constant $C$).
\end{theorem}
\begin{proof}
The proof follows along the lines of its counterpart in the continuous setting. 
 The discrete subspaces of the Helmholtz-type decomposition are $\IVh = \ker b_h$  and $\IWh = \{\w_h\in\IXh \mid a_h(\w_h,\v_h) = 0~ \forall \v_h \in\IVh \}$. 
 Next, we rewrite 
 $
 \IVh =\ker b_h$ as $ \IVh = \ker \div_{\Th}\cap \ker\operatorname{tr}_{\jump{\cdot}} = \ker \tilde b_h$ 
 whereat $\div_{\Th}\colon \IXh \to L^2(\dom)$ is the broken, i.e.\ element-wise, divergence operator and $\operatorname{tr}_{\jump{\cdot}}\colon \IXh\to L^2(\Fh)$ the trace jump operator on the skeleton.
We introduce $\tilde b_h$ as the bilinear form corresponding to these operators  with
\begin{align*}
    \tilde b_h\colon \IXh \times (\IQ_h \times \IF_h) \to \mathbb{R}, \qquad 
    \tilde b_h(\u_h, (q_h,\mu_h) ) & =\tilde d_h (\u_h, q_h) + \tilde n_h (\u_h, \mu_h)        \\
    \tilde d_h(\u_h, q_h )  =\sum_{T\in\Th} \inner{ 
        \div \u_h, q_h }_T,\qquad 
    \tilde n_h(\u_h, \mu_h ) & = \inner{\jump{\u_h},\mu_h}_{\frac12,h,\Fh}. 
\end{align*}
Here, the spaces  $\IQ_h$ and $\IF_h$ are taken as the range of $\IX_h$ under $\div_{\Th}$ and $\operatorname{tr}_{\jump{\cdot}}$, respectively:
$
\IQ_h 
 := 
 \IQ^{p-1}(\Th),~ 
   \IF_h := \{\mu \in L^2(\Fh), \mu \in \IP^p(F),\forall F \in \calF_h \}.
$
Here, $\IQ_h$ is naturally equipped with the $L^2(\dom)$-norm while $\IF_h$ is equipped with the $\norm{\cdot}_{\frac12,h,\Fh}$-norm.     
Now, by construction we see $\ker b_h = \ker \tilde b_h$. Next, we want to obtain inf-sup stability of $\tilde b_h$ in order to control the $\norm{\cdot}_{a_h}$-norm on $\IWh$ by the (scaled) element-wise divergence-terms and the normal jumps. 

For two-fold saddle point problems, see \cite[Thm. 3.1]{zbMATH05942045}, it is known that in order to show inf-sup stability for $\tilde b_h (\cdot,\cdot)$ it suffices to show inf-sup stability of  $\tilde n_h(\cdot,\cdot)$ and inf-sup stability of  $\tilde d_h(\cdot,\cdot)$ on the kernel of $\tilde n_h(\cdot,\cdot)$ (in corresponding norms).

\noindent\underline{1. Inf-sup stability of $\tilde n_h$:} 
Given $\mu_h \in \IF_h$ we construct an element-wise BDM interpolation $\u_h \in \IX_h$, cf. \cite{fortin2013mixed}, where on each facet with aligned elements $T_+$ and $T_-$ we prescribe $\u_h|_{T_{\pm}} \cdot \nx_{\pm} = \pm \frac12 \mu_h$ so that $\jump{\u_h} = \mu_h$ on $\Fh$.
Additionally we demand $(\u_h,\varphi)_T=0$ for all $\varphi \in \calN^{k-2}(T) = [\IP^{p-2}(T)]^d + [\IP^{p-2}(T)]^d\times \bx$ yielding a unique interpolant $\u_h \in \IXh$.
By standard scaling arguments, \cite[Proposition 4.3.5]{alemanmaster}, there holds 
$$
\ol\rho^{-\frac12}\norm{\vel}_{\infty}^{-1} \norm{\u_h}_{a_h} \leq \cbdm{1} \norm{\mu_h}_{\frac12,h,\Fh} = \big( \sum_{T \in \Th}  \norm{h_F^{-\frac12} \mu_h}_{\partial T}^2 \big)^{\frac12}
$$
so that $\forall~\mu_h\in \IF_h$, $\exists~\u_h \in \IX_h$ with 
$$
\tilde n_h(\u_h,\mu_h)
=\inner{\jump{\u_h}, \mu_h}_{\frac12,h,\Fh}
= \norm{\mu_h}_{\frac12,h,\Fh}^2    \geq (\cbdm{1})^{-1}~\ol\rho^{-\frac12}\norm{\vel}_{\infty}^{-1} \norm{\u_h}_{a_h} \norm{\mu_h}_{\frac12,h,\Fh}.
$$

\noindent\underline{2. Inf-sup stability of $\tilde d_h$ on $\ker\tilde n_h$ :} 
Now we need to show inf-sup stability for $d_h(\cdot,\cdot)$ on the kernel of $\tilde n_h(\cdot,\cdot)$.
The kernel is precisely given by the BDM finite element 
$\ker \tilde n_h = \{\u_h\in\IX_h\mid \tilde n_h(\u_h,\mu_h)=0,\ \forall\mu_h\in\IF_h\}= [\IQ^{p}(\Th)]^d\cap\Hdiv$.
Similar to the continuous level we have 
$$\sqrt{\ol{\rho}}^{-2}\norm{\vel_h}_\infty^{-2} \norm{\u_h}_{a_h}^2\leq \sum_{T\in\Th}\norm{\nabla \u_h}_T^2+\norm{\u_h}_T^2 +\norm{\sjump{\u_h} }_{\frac12,h,\Fh}^2  =:\norm{\u_h}_{1,h}^2  \text{ for }\u_h\in \ker \tilde n_h.\vspace*{-0.2cm} $$
where $\sjump{\cdot}$ is the standard DG jump operator, i.e.\ similar to $\bjump{\cdot}$ but without the $\vel$-weighting.    
Then using inf-sup condition results in discrete $H^1$ norms based on a BDM interpolation $\u_h^{\star}=\Pibdm \u^{\star}$ of the $\div$-supremizer $\u^\star$ in $H_0^1(\dom,\mathbb{R}^d)$, see e.g.\ \cite[Proposition 10]{zbMATH01742006}, we obtain $\forall q\in \IQ_h$
\begin{align*}
    \ol{\rho}^{\frac12}\norm{\vel_h}_\infty \!\!\sup_{\u_h \in \ker\tilde n_h}\!\!\frac{\tilde d_h(\u_h, q_h)}{\norm{\u_h}_{a_h}} 
    & \geq \!\sup_{\u_h \in \ker\tilde n_h}\!\!\! \frac{\tilde d_h(\u_h, q_h)}{\norm{\u_h}_{1,h}}
    \geq
    \frac{ \inner{\div \u_h^\star, q_h} }{\norm{\u_h^\star}_{1,h}}
    \geq
   \cbdm{2} \frac{ \inner{\div \u^\star, q_h} }{\norm{\u^\star}_{H^1(\dom)}}
    \geq c_{\operatorname{div},h} \norm{q_h}_\Omega 
\end{align*} 
with $c_{\operatorname{div},h}=\cbdm{2} c_{\operatorname{div}}$, where $\cbdm{2}$ is the inverse of the continuity constant of $\Pibdm$ in $\norm{\cdot}_{1,h}$.    
\\ 
\noindent\underline{3. Inf-sup stability of $\tilde b_h$:} 
We are now able to apply \cite[Thm. 3.1]{zbMATH05942045} to deduce 
\begin{align*}
    \sup_{\u_h \in \IX_h}\frac{\tilde b_h(\u_h, (q_h,\mu_h))}{\norm{\u_h}_{a_h}} 
    \geq c\ol{\rho}^{-\frac12} \norm{\vel}^{-1}_\infty (\norm{q_h}_\Omega + \norm{\mu_h}_{\frac12,h,\Fh} )
    \qquad \forall q\in \IQ_h
\end{align*} 
which allows to control $\norm{\w_h}_{a_h}$ by $\div_{\Th}$ and normal jump terms using \cite[Ch. III, Lemma 4.2]{braess}
\begin{align*}
    \norm{\w_h}_{a_h}^2 \leq C R_{c_s}^2R_\rho \norm{c_s^{-1}\vel}^{2}_\infty( \norm{c_s\sqrt\rho\div\w_h}_\Omega^2 + \norm{c_s\sqrt\rho\jump{\w_h}}_{\frac12,h,\Fh}^2 ) \quad \forall \w_h\in\IW_h.
\end{align*}
We obtain the desired result by applying \eqref{eq:dgnorms}.
\end{proof}

\begin{remark}\label{rem:cea}  
    The previous two results can be combined with even more obvious properties, such as consistency and continuity of $a_h(\cdot,\cdot)$ and $b_h(\cdot,\cdot)$ and interpolation estimates to apply \cref{lem:cea} and obtain the following error bound for the discretization error
    \begin{equation}
        \norm{\u - \u_h}_{\IXh} \leq C h^{l-1}  \Vert \u\Vert_{H^{l}(\dom)}       
    \end{equation}
    with $l = \min\{p+1,m\}$ with $\u\in H^{m}(\dom,\mathbb{R}^d),~m\geq 2,$ the exact solution.
\end{remark}  
\begin{remark}\label{rem:carryover}  
    The previous remark as well as the previous two results for \eqref{M4} directly carry over for \eqref{M3}. In the case of \eqref{M3} the proofs simplify as the normal jump terms vanish. Furthermore for $H^1$-conforming finite element spaces $\IXh$ that in combination with $\IQ_h = \div \IXh$ correspond to inf-sup stable Stokes elements, both statements remain true with even further simplifications in the proofs. 
\end{remark} 
\begin{remark}[Robustness] \eqref{M3} and \eqref{M4} are locking-free and gradient-robust in the sense of \cref{def:lf,def:gr}.  
    Applying an element-wise BDM interpolator $\Pibdmdg$, cf. \cite{fortin2013mixed}, yields for any $\v^\ast\in\IV$ that $\Pibdmdg \v^\ast \in \IV_h$ and $\norm{\Pibdmdg \v^\ast}_{\IXh} \leq \cbdmdg \norm{\v^\ast}_{\IXh}$ and the locking-free property follows with
    $$
    \inf_{\v_h\in \IVh }  \norm{\v-\v_h}_{\IXh} 
    \leq \norm{\v-\Pibdmdg \v}_{\IXh}   
    \leq \norm{\v-\v_h^\ast}_{\IXh} + \norm{ \Pibdmdg (\v-\v_h^\ast) }_{\IXh}
    \leq (1+\cbdmdg) \norm{\v-\v_h^\ast}_{\IXh}
    $$
    for any $\v\in\IV$.
    With $\div_{\Th} \v_h' = 0$ and $\jump{\v_h'}=0$ for all $\v_h'\in\IVh$ for \eqref{M3} and \eqref{M4} we have by partial integration $\inner{\nabla\phi, \v_h'} = 0$ for any $\phi\in H^1(\dom)$     
    which implies $\u_h\in\IWh$ and finally gradient-robustness.
\end{remark}

\section{Numerical examples} \label{sec:numerics}
We choose $\bbf$ such that the exact solution of \eqref{eq:strongpde} is given by
    $\u_\text{sol}=\sin(\pi x)\cos(\pi y)( -y , x )^T$
and solve on a disc with radius 1. The background flow is chosen as 
$\vel=0.1(-y,x)^t$
and $\rho=c_s=1,\lambda_\vel=10p^2,\ \lambda_\nx=100p^2$.
The results are shown in \Cref{fig:hconv} for varying mesh sizes and polynomial degree $p=1,2,3,4$.
The methods \eqref{M3} and \eqref{M4} show convergence of order $\mathcal O(h^{p+1/2})$ in the $L^2$-norm.
Effects of the different aspects of robustness discussed in \Cref{sec:robustness} can also be observed.
The method \eqref{M1} is not volume-locking free for low order which we observe here for $p=1,2$.
The method \eqref{M2} shows the same rates, however shows worse error than the other methods, since it lacks gradient-robustness.
\begin{figure}[!ht]
    \resizebox{\linewidth}{!}{
    \begin{tikzpicture} [spy using outlines={circle, magnification=4, size=2cm, connect spies}]
        \begin{groupplot}[%
          group style={%
            group name={my plots},
            group size=4 by 1,
            ylabels at=edge left,
          },
        legend style={
            legend columns=6,
            at={(-1.3,-0.2)},
            anchor=north,
            draw=none
        },
        xlabel={$h$},
        xlabel style={yshift=1.5ex},
        ylabel={$L^2$-error},
        ymajorgrids=true,
        grid style=dashed,
        cycle list name=paulcolors4,
        width=6cm,height=5cm,
        title style={yshift=-1.5ex},
        ]      
        \nextgroupplot[ymode=log,xmode=log,x dir=reverse, title={$p=1$}]
            \foreach \p in {1}{
                \addplot+[discard if not={p}{\p}] table [x=h, y=errorH1, col sep=comma] {results/hconv.csv};
                \pgfplotsset{cycle list shift=1}
                \addplot+[discard if not={p}{\p}] table [x=h, y=errorHdiv, col sep=comma] {results/hconv.csv};
                \addplot+[discard if not={p}{\p}] table [x=h, y=errorDG, col sep=comma] {results/hconv.csv};
            }
            \logLogSlope{1}{1}{0.8}{1.5}{black,dashed};
        \nextgroupplot[ymode=log,xmode=log,x dir=reverse, title={$p=2$}]
            \foreach \p in {2}{
                \addplot+[discard if not={p}{\p}] table [x=h, y=errorH1, col sep=comma] {results/hconv.csv};
                \addplot+[discard if not={p}{\p}] table [x=h, y=errorH1pp, col sep=comma] {results/hconv.csv};
                \addplot+[discard if not={p}{\p}] table [x=h, y=errorHdiv, col sep=comma] {results/hconv.csv};
                \addplot+[discard if not={p}{\p}] table [x=h, y=errorDG, col sep=comma] {results/hconv.csv};
            }
            \logLogSlope{1}{1}{0.8}{2.5}{black,dashed};
        \nextgroupplot[ymode=log,xmode=log,x dir=reverse, title={$p=3$}]
            \foreach \p in {3}{
                \addplot+[discard if not={p}{\p}] table [x=h, y=errorH1, col sep=comma] {results/hconv.csv};
                \addplot+[discard if not={p}{\p}] table [x=h, y=errorH1pp, col sep=comma] {results/hconv.csv};
                \addplot+[discard if not={p}{\p}] table [x=h, y=errorHdiv, col sep=comma] {results/hconv.csv};
                \addplot+[discard if not={p}{\p}] table [x=h, y=errorDG, col sep=comma] {results/hconv.csv};
            }
            \logLogSlope{1}{1}{0.8}{3.5}{black,dashed};
        \nextgroupplot[ymode=log,xmode=log,x dir=reverse, title={$p=4$}]
            \foreach \p in {4}{
                \addplot+[discard if not={p}{\p}] table [x=h, y=errorH1, col sep=comma] {results/hconv.csv};
                \addplot+[discard if not={p}{\p}] table [x=h, y=errorH1pp, col sep=comma] {results/hconv.csv};
                \addplot+[discard if not={p}{\p}] table [x=h, y=errorHdiv, col sep=comma] {results/hconv.csv};
                \addplot+[discard if not={p}{\p}] table [x=h, y=errorDG, col sep=comma] {results/hconv.csv};
            }
            \logLogSlope{1}{1}{0.8}{4.5}{black,dashed};
        \legend{H1 \eqref{M1}, H1 pseudo-pressure \eqref{M2}, $\Hdiv$ \eqref{M3}, DG \eqref{M4}, $\mathcal O(h^{p+1/2})$}
        \end{groupplot}
    \end{tikzpicture}}
    \vspace*{-0.5cm}  
    \caption{Convergence study in terms of mesh size $h$ for polynomial degrees $p=1,2,3,4$.}
    \label{fig:hconv}
\end{figure}
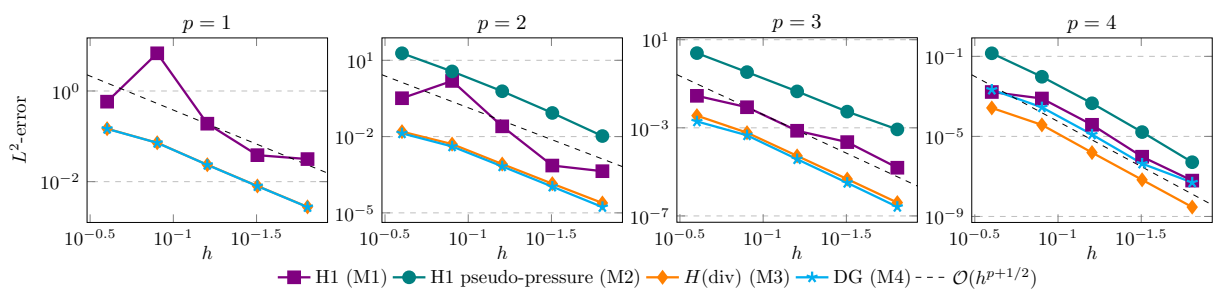

\bibliographystyle{siam}
\bibliography{proc}

\end{document}